\documentclass[oneside,english]{amsart}
\pdfoutput=1
\usepackage[T1]{fontenc}
\usepackage[latin9]{inputenc}
\usepackage{geometry}
\usepackage{babel}
\usepackage{mathtools}
\usepackage{amsbsy}
\usepackage{amstext}
\usepackage{amsthm}
\usepackage{amssymb}
\usepackage{graphicx}
\usepackage{verbatim}
\usepackage[unicode=true,pdfusetitle, bookmarks=true,bookmarksnumbered=false,bookmarksopen=false, breaklinks=false,pdfborder={0 01},backref=false] {hyperref}
\usepackage{cleveref}
\usepackage[all]{xy}
\usepackage{xcolor}
\usepackage{soul}
\makeatletter
\numberwithin{equation}{section}
\numberwithin{figure}{section}
\theoremstyle{plain}
\newtheorem{thm}{\protect\theoremname}[section]
\theoremstyle{plain}
\newtheorem{cor}[thm]{\protect\corollaryname}
\theoremstyle{plain}
\newtheorem{lem}[thm]{Lemma}

\theoremstyle{definition}

\theoremstyle{definition}

\theoremstyle{definition}
\newtheorem{qn}[thm]{Question}
\newtheorem{rem}[thm]{\protect\remarkname}

\newcommand\Z{\mathbb Z}

\def\gr{\operatorname{gr}}

\title{Topologically and rationally slice knots}
\author{Jennifer Hom}
\address{School of Mathematics, Georgia Institute of Technology}
\email{hom@math.gatech.edu}
\author{Sungkyung Kang}
\address{Center for Geometry and Physics, Institute for Basic Science}
\email{sungkyung38@icloud.com}
\author{JungHwan Park}
\address{Department of Mathematical Sciences, Korea Advanced Institute for Science and Technology}
\email{jungpark0817@kaist.ac.kr}

\makeatother

\providecommand{\corollaryname}{Corollary}
\providecommand{\definitionname}{Definition}
\providecommand{\remarkname}{Remark}
\providecommand{\theoremname}{Theorem}
\renewcommand{\setminus}{-}
\begin{document}

\begin{abstract}
A knot in $S^3$ is topologically slice if it bounds a locally flat disk in $B^4$. A knot in $S^3$ is rationally slice if it bounds a smooth disk in a rational homology ball. We prove that the smooth concordance group of topologically and rationally slice knots admits a $\mathbb{Z}^\infty$ subgroup. All previously known examples of knots that are both topologically and rationally slice were of order two. As a direct consequence, it follows that there are infinitely many topologically slice knots that are strongly rationally slice but not slice.
\end{abstract}
\maketitle

\section{Introduction}
The smooth concordance group of topologically slice knots contains a $\Z^\infty \oplus (\Z/2\Z)^\infty$ subgroup; the existence of a $\Z^\infty$ subgroup was first proved by Endo in 1995 using gauge theory \cite{Endo}, while the existence of a $(\Z/2\Z)^\infty$ subgroup was not proved until nearly 20 years later, using Heegaard Floer homology \cite{HeddenKimLiv}. Both of these proofs rely on Freedman's landmark result \cite{Freedman} that any knot with Alexander polynomial 1 is topologically slice.

Similarly, the smooth concordance group of rationally slice knots also contains a $\Z^\infty \oplus (\Z/2\Z)^\infty$ subgroup. However, our understanding of rationally slice knots proceeded in the opposite direction: the proof  of the $(\Z/2\Z)^\infty$ subgroup \cite{ChaMAMS} preceded the proof of the $\Z^\infty$ subgroup~\cite{HKPS} by 15 years. For their obstructive side, the former proof relies on J.\ Levine's algebraic concordance group \cite{JLevine} while the latter relies on involutive Heegaard Floer homology \cite{HendricksManolescu}. For their constructive side, they use the fact that the figure-eight knot is rationally slice and generalize its proof. This fact was observed by Cochran using work of Fintushel-Stern~\cite{FintushelStern}.

\vspace{.2cm}
What about knots that are both topologically and rationally slice? 
\begin{itemize}
    \item Rationally slice but not slice: The knots in \cite{ChaMAMS} are not topologically slice, as they are not even algebraically slice. Similarly, the generating set for the $\Z^\infty$ subgroup in \cite{HKPS} are not algebraically slice either. 
    \item Topologically slice but not slice: The knots representing the $(\Z/2\Z)^\infty$ subgroup of \cite{HeddenKimLiv} can be shown to be rationally slice by a straightforward modification of the proof of \cite[Theorem 4.16]{ChaMAMS}. The generating set for Endo's family~\cite{Endo}
 (the pretzel knots $P(-2k-1, 4k+1, 4k+3)$ for $k \geq 1$) are not rationally slice. There are more topologically slice knots in the literature and many of them are known to be not rationally slice (see e.g.~\cite{Hom:2015-1, Oz-St-Sz:2017-1, Kim-Park:2018-1, Feller-Park-Ray:2019-1, DHMT:2021-1}) and some of them are not known to be rationally slice or not (see e.g.~\cite{Hedden-Kirk:2012-1, CHH:2013-1, Cha-Kim:2021-1, Cha:2021-1, KKK:2022-1}).
\end{itemize}
\vspace{.2cm}

Note that all previously known examples of knots that are both topologically and rationally slice have \emph{order two} in the smooth concordance group. The goal of this paper is to show that the smooth concordance group of topologically and rationally slice knots does in fact contain a $\Z^\infty \oplus (\Z/2\Z)^\infty$ subgroup. For any knot $K$, let $\mathrm{Wh}^+(K)$ denote the positive Whitehead double of $K$, and $K_{p,-1}$ denote the $(p,-1)$-cable of $K$.

\begin{thm}\label{thm:main}
Denote by the right-handed trefoil and the figure-eight knot by $T$ and $E$, respectively. Then every knot in the family 
\[
\left\{ K_n:=(\mathrm{Wh}^+(T)\# E)_{2n+1,-1} \# -(\mathrm{Wh}^+(T))_{2n+1,-1} \# -E_{2n+1,-1} \mid n>0 \right\}
\]
is topologically and rationally slice, and has infinite order in the smooth concordance group. Furthermore, the given family admits a linearly independent infinite subset of the smooth concordance group.
\end{thm}

As mentioned in Theorem \ref{thm:main}, the knots $K_n$ are topologically and rationally slice. Indeed, up to topological concordance, we may ignore $\mathrm{Wh}^+(T)$, resulting in $E_{2n+1,-1} \# -E_{2n+1,-1}$, which is clearly slice. Similarly, up to rational concordance, we may ignore $E$, which also results in a slice knot. Our family of knots exploits the fact that cabling does not induce a homomorphism on the concordance group; see~\cite{HeddenPinzon} for further discussion of satellite operators and the concordance group.

We may naturally interpret our main theorem as follows. Let $\mathcal{T}$ be the smooth concordance group of topologically slice knots and $\mathcal{T}_\mathbb{Q}$ be the smooth rational concordance group of topologically slice knots. Since two concordant knots are rationally concordant, we have the following natural surjective homomorphism:
$$ \psi_\mathcal{T}\colon \mathcal{T} \to \mathcal{T}_\mathbb{Q}.$$
Note that the main theorem of \cite{HeddenKimLiv} implies that $\ker \psi_\mathcal{T}$ contains infinitely many order two elements. As a direct corollary of our main theorem, we have the following:
\begin{cor}\label{cor:main}
The group $\ker \psi_\mathcal{T}$ contains a subgroup isomorphic to $\mathbb{Z}^\infty$. 
\end{cor}

Furthermore, the knots in \Cref{thm:main} can be used to prove the existence of a knot which is topologically slice and strongly rationally slice but not slice. Strongly rationally slice knots are knots $K$ which bound a smoothly embedded disk $\Delta$ in a rational homology ball $X$ such that the inclusion induces an isomorphism
\[
H_1(S^3 \setminus K;\mathbb{Z}) \xrightarrow{\cong} H_1(X \setminus \Delta; \mathbb{Z})/\text{torsion}.
\]
The question of distinguishing strongly rationally slice knots from slice knots is very subtle; the first example of such a knot was found in \cite{kang2022torsion}. By directly applying the proof of \cite[Theorem 1.4]{kang2022torsion} to the knots in \Cref{thm:main}, we immediately get the following corollary.

\begin{cor}
   Let $\{K_n\}$ be the knots in \Cref{thm:main}.
    Then each $(K_n)_{2,1}$ is topologically and strongly rationally slice, and has infinite order in the smooth concordance group.
\end{cor}



Recall that there exists a natural surjective homomorphism $\psi\colon \mathcal{C} \to \mathcal{C}_\mathbb{Q}$ defined from the smooth concordance group to the smooth rational concordance group. As its counterpart in the topological category, we can also consider the map $$\psi^\text{top} \colon \mathcal{C}^\text{top} \to \mathcal{C}_\mathbb{Q}^\text{top}$$ defined from the topological concordance group to the topological rational concordance group. Even though we know that there are subgroups $\mathbb{Z}^\infty \oplus (\mathbb{Z}/2\mathbb{Z})^\infty$ in $\ker \psi$ and $(\mathbb{Z}/2\mathbb{Z})^\infty$ in $\ker \psi^\text{top}$ \cite{ChaMAMS, HKPS}, the existence of an infinite order element in $\ker \psi^\text{top}$ is still not known. Hence, we ask:

\begin{qn} Does there exist an infinite order element in $\ker \psi^\text{top}$?
\end{qn}

Lastly, we make a remark that the positive Whitehead double of the figure-eight knot is also a good candidate that is likely to have infinite order in $\ker \psi_\mathcal{T}$. To the best of the authors' knowledge it is not even known if this knot is slice or not. Here, we ask:

\begin{qn} Is the subgroup generated by $\mathrm{Wh}^+(E)$ in $\ker \psi_\mathcal{T}$ isomorphich to $\mathbb{Z}$? More specifically, is $\mathrm{Wh}^+(E)$ slice?\end{qn}

\subsection{Proof outline} For the reader's convenience, we now provide a brief overview of the proof strategy. The proof relies on bordered and involutive knot Floer homology, using several tricks and techniques along the way to ease the calculations as much as possible.
\begin{enumerate}
    \item Up to $\iota_K$-local equivalence, we can replace $\mathrm{Wh}^+(T)$ by $T \# E$ (Lemma \ref{lem:iotaKD}).
    \item\label{it:iT} There is a (non-involutive) inclusion map 
\[
i_T : \widehat{CFD}(S^3 \setminus T) \rightarrow \widehat{CFD}(S^3 \setminus (T\#E)).
\]
Using Zhan's bordered Floer homology calculator \cite{bfh_python}, we determine that certain type-D morphisms are nullhomotopic  (Lemmas \ref{lem:tensorlem} and \ref{lem:trefoillem}), allowing us to conclude that $i_T$ induces an almost $\iota_K$-local map from the almost $\iota_K$-complex of the $(2n+1,-1)$-cable of $T$ to that of $T\#E$ (or equivalently that of $\mathrm{Wh}^+(T)$, as in Equation \ref{eq:TleqD}).
    \item We combine the result from \eqref{it:iT} with results from \cite{kang2022torsion} to compare our knots $K_n$ to certain standard complexes $C_n$, with respect to the partial order on the horizontal almost $\iota_K$-local equivalence group (described in Section \ref{sec:horzalmost}). Using properties of how the $C_n$ interact with the partial order, we arrive at the desired linear independence result.
\end{enumerate}

\subsection*{Acknowledgements} The authors would like to thank Gary Guth for his help on using Zhan's bordered Floer homology calculator. JH was partially supported by NSF grant DMS-2104144 and a Simons Fellowship. SK was supported by the Institute for Basic Science (IBS-R003-D1). JP was partially supported by Samsung Science and Technology Foundation (SSTF-BA2102-02) and the POSCO TJ Park Science Fellowship.

\section{The horizontal almost $\iota_K$-local equivalence group}\label{sec:horzalmost}
We assume that the reader is familiar with involutive knot Floer homology \cite{HendricksManolescu} and the $\iota_K$-local equivalence group \cite{HMZ}. See \cite[Section 2]{HKPS} for an expository overview of these concepts. In this paper, we focus on the horizontal almost $\iota_K$-local equivalence group, which has the advantage that, modulo the image of the figure eight knot, it is totally ordered. What follows is a short summary of \cite[Section 2]{kang2022torsion}.

A \emph{horizontal almost $\iota_K$-complex} is a pair $(C,\iota)$ of a bigraded complex $C$ of finitely generated free modules over $\mathbb{F}_2[U]$ and a chain homotopy equivalence $\iota\colon\widehat{C}\rightarrow \widehat{C}$, where $\widehat{C}$ is the hat-flavored truncation of $C$, such that the following conditions are satisfied. 
\begin{itemize}
    \item The formal variable $U$ has bigrading $(-2,0)$.
    \item $U^{-1}C\simeq \mathbb{F}_2[U,U^{-1}]$
    \item $\iota$ is skew-graded, i.e. maps a $(a,b)$-bigraded element to a $(b,a)$-bigraded one.
    \item $\Phi\iota\Phi\iota \sim \iota\Phi\iota\Phi$.
    \item $\iota^2 \sim 1+\Phi\iota\Phi\iota$.
    \item There exists a chain map $f\colon C\rightarrow C$ whose hat-flavored truncation is homotopic to $\iota\Phi\iota$.
\end{itemize}
Here, $\Phi$ denotes the formal derivative of the differential of $C$ with respect to the formal variable $U$, which is a chain map which is well-defined up to homotopy.

A degree-preserving chain map $f\colon C\rightarrow D$ between horizontal almost $\iota_K$-complexes $(C,\iota_C)$ and $(D,\iota_D)$ is \emph{almost $\iota_K$-local map} if it satisfies the following conditions.
\begin{itemize}
    \item $f$ is \emph{local}, i.e.\ the localized map $U^{-1}f\colon U^{-1}C\rightarrow U^{-1}D$ is a homotopy equivalence.
    \item $\iota_C \widehat{f} \sim \widehat{f} \iota_D$, where $\widehat{f}$ is the hat-flavored truncation of $f$.
\end{itemize}
Furthermore, if there exist almost $\iota_K$-local maps $f\colon C\rightarrow D$ and $g:D\rightarrow C$, then we say that $(C,\iota_C)$ and $(D,\iota_D)$ are \emph{almost $\iota_K$-locally equivalent}. 

We denote the almost $\iota_K$-local equivalence classes of horizontal almost $\iota_K$-complexes by $\mathfrak{I}^U _K$. We endow this set with a tensor product operation $\otimes$ as $(C,\iota_C)\otimes (D,\iota_D)=(C\otimes D,\iota_{C\otimes D})$, where $\iota_{C\otimes D}$ is defined as $\iota_{C\otimes D} = \iota_C \otimes \iota_D + \Phi\iota_C \otimes \iota_D \Phi$. Although this operation is nonsymmetric, the group $\mathfrak{I}^U_K$ is indeed abelian~\cite[Proposition 2.6]{kang2022torsion}.

The involutive knot Floer homology package associates to a knot $K$ a well-defined element in $\mathfrak{I}^U_K$, given by the chain complex $CFK^-(S^3,K)$, together with the $\iota_K$-action on $\widehat{CFK}(S^3,K)$. For simplicity, we denote this element as $[K]$. It is clear that this defines a group homomorphism 
\[
\mathcal{C}\rightarrow \mathfrak{I}^U_K,
\]
where $\mathcal{C}$ is the smooth concordance group. The image of the unknot is denoted by $0 \in \mathfrak{I}^U_K$.

Notice that the existence of almost $\iota_K$-local map gives a partial order on $\mathfrak{I}^U_K$. In particular, we say that $C \le D$ if there exists an almost $\iota_K$-local map $C\rightarrow D$. It turns out that this is actually a total order $\mathfrak{I}^U_K$, modulo the figure-eight complex $[E]$, i.e.\ the 2-torsion element induced by the involutive knot Floer homology of the figure-eight knot $E$. In other words, two elements $C,D\in \mathfrak{I}^U_K$ are incomparable if and only if $C=D+[E]$~\cite[Theorem 2.11]{kang2022torsion}.  We will exploit this ordering to prove the main theorem.

\section{Lemmas from involutive bordered Floer homology}

We assume that the reader is familiar with standard materials in bordered Heegaard Floer homology, in particular the materials in \cite{lipshitz2018bordered}.

For simplicity, from now on, we will denote the $\infty$-framed solid torus, as a bordered manifold, as $T_\infty$. Given a pattern $P\subset T_\infty$, we can define the minus-flavored type-A module $CFA^-(T_\infty,P)$, which is a type-A structure over the torus algebra $\mathcal{A}(T^2)$ with the coefficient ring $\mathbb{F}_2[U]$. We denote its truncation by $U=0$ as $\widehat{CFA}(T_\infty,P)$. Furthermore, we can also remove a tubular neighborhood of $P$ from $T_\infty$ and endow the newly created torus boundary with the $0$-framing; this defines a bordered manifold with two torus boundaries, whose type-DA bordered Floer homology is denoted as $\widehat{CFDA}(T_\infty \setminus P)$. Note that we have 
\[
\widehat{CFD}(S^3 \setminus P(K))\simeq \widehat{CFDA}(T_\infty \setminus P)\boxtimes \widehat{CFD}(S^3 \setminus K)
\]
via standard gluing formulae.

We will denote the longitudinal knot inside $T_\infty$ by $\nu$. In particular, we have a gluing formula \cite[Theorem 11.29]{lipshitz2018bordered}:
\[
CFK^-(S^3,K) \simeq CFA^-(T_\infty,\nu) \boxtimes \widehat{CFD}(S^3 \setminus K),
\]
where $CFK^-$ denotes the truncation of $CFK_{UV}(S^3,K)$ by $V=0$. Note that we can also apply the gluing formula to get an identification
\[
CFA^-(T_\infty ,P)\simeq CFA^-(T_\infty,\nu)\boxtimes \widehat{CFDA}(T_\infty \setminus P).
\]
Given a bordered manifold $M$ with one torus boundary, its type-D bordered involution takes the form
\[
\iota_M : \widehat{CFDA}(\mathbf{AZ})\boxtimes \widehat{CFD}(M)\rightarrow \widehat{CFD}(M),
\]
where $\mathbf{AZ}$ denotes the Auroux-Zarev piece, defined in \cite[Section 4]{lipshitz2011heegaard}. Although $\iota_M$ is not well-defined up to homotopy due to the lack of naturality in bordered Floer homology, it is still a homotopy equivalence, and we will not need its uniquenss anyway. We denote the set of homotopy classes of all possible type-D bordered involutions of $M$ by $\mathbf{Inv}_D(M)$.

The following lemma follows from the proof of \cite[Theorem 4.5]{kang2022torsion}. Its proof is straightforward from the discussions preceding the proof of \cite[Theorem 1.2]{kang2022involutive}, so we omit it for the sake of simplicity.

\begin{lem} \label{lem:KangParkLem2}
    Given two knots $K_1,K_2$, let $f\colon \widehat{CFD}(S^3 \setminus K_1) \rightarrow \widehat{CFD}(S^3 \setminus K_2)$ be a type-D morphism. Given a pattern $P \subset T_\infty$, consider the induced type-D morphism $P(f)\colon \widehat{CFD}(S^3 \setminus P(K_1)) \rightarrow \widehat{CFD}(S^3 \setminus P(K_2))$, defined as
    \[
    \begin{split}
        \widehat{CFD}(S^3 \setminus P(K_1)) &\simeq \widehat{CFDA}(T_\infty \setminus P) \boxtimes \widehat{CFD}(S^3 \setminus K_1) \\
        &\xrightarrow{\mathrm{id}_{\widehat{CFDA}(T_\infty \setminus P)}\boxtimes f} \widehat{CFDA}(T_\infty \setminus P) \boxtimes \widehat{CFD}(S^3 \setminus K_2) \simeq \widehat{CFD}(S^3 \setminus P(K_2)).
    \end{split}
    \]
    Then we have $$\iota_{S^3 \setminus P(K_2)} \circ \left(\mathrm{id}_{\widehat{CFDA}(\mathbf{AZ})}\boxtimes  P(f)\right) \circ \iota^{-1}_{S^3 \setminus P(K_1)} \sim P\left(\iota_{S^3 \setminus K_2} \circ \left(\mathrm{id}_{\widehat{CFDA}(\mathbf{AZ})}\boxtimes  f\right) \circ \iota^{-1}_{S^3 \setminus K_1}\right)$$ for suitable choices of bordered involutions $\iota_{S^3 \setminus P(K_i)} \in \mathbf{Inv}_D(S^3 \setminus P(K_i))$ and $\iota_{S^3 \setminus K_i}\in \mathbf{Inv}_D(S^3 \setminus K_i)$ for each $i\in \{1,2\}$.
\end{lem}

\section{Proof of the main theorem}

We start with several explicit computations of bordered Floer homology modules. Given any knot $K$ in $S^3$, the type-D module $\widehat{CFD}(S^3 \setminus K)$ can be easily computed from the knot Floer chain complex $CFK_{UV}(S^3,K)$. For example, for the right-handed trefoil $T$, the type-D module $\widehat{CFD}(S^3 \setminus T)$ for the 0-framed complement of $T$ can be described as follows:
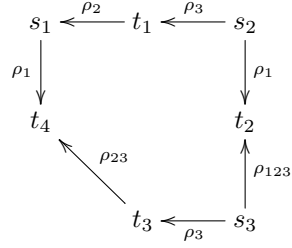
\[
\xymatrix{
s_1 \ar[d]_{\rho_1} & t_1 \ar[l]_{\rho_2} & s_2 \ar[l]_{\rho_3} \ar[d]^{\rho_1} \\
t_4 & & t_2 \\
& t_3 \ar[ul]_{\rho_{23}} & s_3 \ar[l]^{\rho_3} \ar[u]_{\rho_{123}}
}
\]
Similarly, we can also describe $\widehat{CFD}(S^3 \setminus E)$ for the 0-framed complement of the figure-eight knot $E$, as shown below. We will call its summand generated by $a,b,c,e,y_1,y_2,y_3,y_4$ as the square-module, and denote it by $S$.
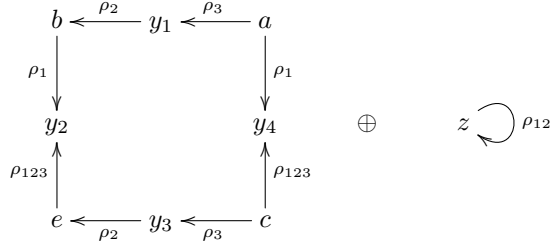
\[
\xymatrix{
b \ar[d]_{\rho_1} & y_1 \ar[l]_{\rho_2} & a \ar[l]_{\rho_3} \ar[d]^{\rho_1} & & \\
y_2 & & y_4 & \oplus & z \ar@(ur,dr)^{\rho_{12}} \\
e \ar[u]^{\rho_{123}} & y_3\ar[l]^{\rho_2} & c\ar[l]^{\rho_3} \ar[u]_{\rho_{123}} & &
}
\]

We will also need the (hat-flavored) type-A module $\widehat{CFA}(T_\infty,P_{2n+1,-1})$ of the $(2n+1,-1)$-cabling pattern in the $\infty$-framed solid torus $T_\infty$, which is shown below. Note that $CFA^-(T_\infty,P_{2n+1,-1})$ was computed originally in \cite[Lemma 8.3]{Oz-St-Sz:2017-1}; we simply truncated their computation by taking $U=0$. 
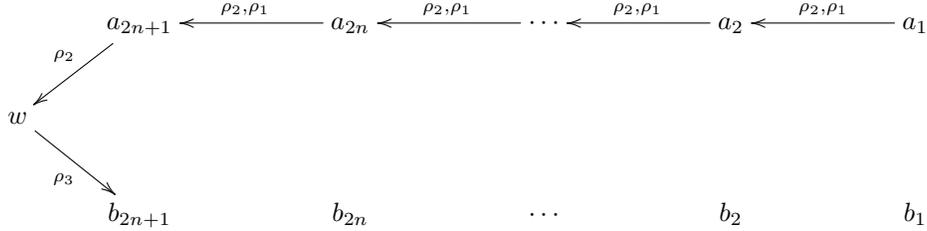
\[
\xymatrix{
& a_{2n+1} \ar[ld]_{\rho_2} && a_{2n}\ar[ll]_{\rho_2,\rho_1} && \cdots\ar[ll]_{\rho_2,\rho_1} && a_2\ar[ll]_{\rho_2,\rho_1} && a_1\ar[ll]_{\rho_2,\rho_1}\\
w \ar[rd]_{\rho_3} & && && && && \\
& b_{2n+1} && b_{2n} && \cdots && b_2 && b_1
}
\]

Based on these computations, we will now prove two computational lemmas, which play a crucial role in the proof of \Cref{thm:main}.

\begin{lem} \label{lem:tensorlem}
Any type-D morphism from $\widehat{CFD}(S^3 \setminus T)$ to the square-summand $S$ becomes nullhomotopic after box-tensored with the identity morphism of $\widehat{CFA}(T_\infty,P_{2n+1,-1})$ for any positive integer $n$.
\end{lem}
\begin{proof}

Zhan's bordered Floer homology calculator \cite{bfh_python} tells us that the space $H_\ast \mathrm{Mor}(\widehat{CFD}(S^3 \setminus T),S)$ of homotopy classes of type-D morphisms from $\widehat{CFD}(S^3 \setminus T)$ to $S$ is six-dimensional, generated over $\mathbb{F}_2$ by $f_1,f_2,f_3,g_1,g_2,g_3$, which are defined as follows.

\[
\begin{split}
    f_1 &: s_2 \mapsto e,\,t_2 \mapsto \rho_{23} y_2, \\
    f_2 &: s_1 \mapsto e,\,s_2 \mapsto c,\,t_2 \mapsto \rho_{23}y_4,\,t_4 \mapsto \rho_{23}y_2, \\
    f_3 &: s_2 \mapsto b,\,s_3 \mapsto e,\,t_2 \mapsto \rho_{23}y_2, \\
    g_1 &: s_3 \mapsto \rho_3 y_1, \\
    g_2 &: s_3 \mapsto \rho_3 y_3, \\
    g_3 &: s_1 \mapsto \rho_1 y_4.
\end{split}
\]
Thus, to prove the lemma, we only have to show that these six morphisms become nullhomotopic after box-tensored with $\mathrm{id}_{\widehat{CFA}(T_\infty,P_{2n+1,-1})}$. For simplicity, we will use the following convention: a chain complex $C$ over $\mathbb{F}_2$ has an acyclic summand $(\partial \colon a\rightarrow b)$ if $a$ and $b$ generate a direct summand of $C$ and $\partial a = b$.

We start with the map $f_1$. The only simple tensors (of basis elements) on which $\mathrm{id}_{\widehat{CFA}(T_\infty,P_{2n+1,-1})}\boxtimes f_1$ takes nontrivial values are $w\otimes s_2$ and $a_{2n+1}\otimes t_2$. In particular, we have 
\[
\left(\mathrm{id}_{\widehat{CFA}(T_\infty,P_{2n+1,-1})}\boxtimes f_1\right)(w\otimes s_2) = w\otimes e\quad\text{and}\quad \left(\mathrm{id}_{\widehat{CFA}(T_\infty,P_{2n+1,-1})}\boxtimes f_1\right)(a_{2n+1}\otimes t_2) = b_{2n+1}\otimes y_2.
\]
However, $w\otimes s_2$ is contained in the acyclic summand $(\partial \colon w\otimes s_2 \rightarrow b_{2n+1}\otimes t_1)$ and $b_{2n+1}\otimes y_2$ is contained in the acyclic summand $(\partial \colon a_{2n}\otimes y_3 \rightarrow b_{2n+1}\otimes y_2)$. Thus $\mathrm{id}_{\widehat{CFA}(T_\infty,P_{2n+1,-1})}\boxtimes f_1$ is nullhomotopic.

For $f_2$, the map $\mathrm{id}_{\widehat{CFA}(T_\infty,P_{2n+1,-1})}\boxtimes f_2$ takes nontrivial values only on $w\otimes s_1$, $a_{2n+1} \otimes t_1$ (which are both mapped to $w\otimes e$), and $w\otimes s_2$. But $w\otimes e$ and $w\otimes s_2$ are contained in the acyclic summands $(\partial \colon a_{2n+1}\otimes y_3\rightarrow w\otimes e)$ and $(\partial \colon w\otimes s_2 \rightarrow b_{2n+1}\otimes t_1)$, respectively. Thus $\mathrm{id}_{\widehat{CFA}(T_\infty,P_{2n+1,-1})}\boxtimes f_2$ is nullhomotopic.

Similarly, the map $\mathrm{id}_{\widehat{CFA}(T_\infty,P_{2n+1,-1})}\boxtimes f_3$ takes nontrivial values on $w\otimes s_2$ and $w\otimes s_3$, which are contained in the acylic summands $(\partial \colon w\otimes s_2\rightarrow b_{2n+1}\otimes t_1)$ and $(\partial \colon w\otimes s_3\rightarrow b_{2n+1}\otimes t_3)$. Hence $\mathrm{id}_{\widehat{CFA}(T_\infty,P_{2n+1,-1})}\boxtimes f_3$ is also nullhomotopic.

For the maps $g_1$ and $g_2$, the only simple tensor on which $\mathrm{id}_{\widehat{CFA}(T_\infty,P_{2n+1,-1})}\boxtimes g_i$ takes a nontrivial value for $i=1,2$ is $w\otimes s_3$, on which $\mathrm{id}_{\widehat{CFA}(T_\infty,P_{2n+1,-1})}\boxtimes g_i$ takes the value $b_{2n+1}\otimes y_1$ for $i=1$ and $b_{2n+1}\otimes y_3$ for $i=2$. But the former is contained in the acyclic summand $(\partial \colon w\otimes a\rightarrow b_{2n+1}\otimes y_1)$ and the latter is contained in the acyclic summand $(\partial \colon w\otimes c\rightarrow b_{2n+1}\otimes y_3)$. Hence $\mathrm{id}_{\widehat{CFA}(T_\infty,P_{2n+1,-1})}\boxtimes g_i$ is nullhomotopic for $i=1,2$.

It remains to show that $\mathrm{id}_{\widehat{CFA}(T_\infty,P_{2n+1,-1})}\boxtimes g_3$ is nullhomotopic. The complete list of its nontrivial values is given as follows:
\[
a_1 \otimes t_1 \mapsto a_2 \otimes y_4, \cdots ,\, a_{2n} \otimes t_1 \mapsto a_{2n+1}\otimes y_4.
\]
But for each $i\in \{1,\cdots,2n\}$, we have that  $a_i \otimes t_1$ is contained in the acyclic summand $(\partial \colon a_i \otimes t_1 \rightarrow a_{i+1}\otimes t_4)$, and thus $\mathrm{id}_{\widehat{CFA}(T_\infty,P_{2n+1,-1})}\boxtimes g_3$ is also nullhomotopic.
\end{proof}

Recall that bordered Heegaard Floer homology comes with gradings by nonabelian groups; given a bordered 3-manifold $Y$ with boundary $\mathcal{Z}$, the associated type-D module $\widehat{CFD}(Y)$ is graded by a transitive $G(\mathcal{Z})$-set. Hence the notion of degree-preserving endomorphisms of $\widehat{CFD}(Y)$ is well-defined. See \cite[Chapter 10]{lipshitz2018bordered} for a detailed description of gradings in bordered Floer homology.

\begin{lem}\label{lem:trefoillem}
    Let $f$ be a degree-preserving endomorphism of $\widehat{CFD}(S^3 \setminus T)$. Then $f$ is either nullhomotopic or homotopic to the identity morphism.
\end{lem}
\begin{proof}
    One can use Zhan's bordered Floer homology calculator to show that the space $H_\ast \mathrm{End}(\widehat{CFD}(S^3 \setminus T))$ is six-dimensional, generated by the identity morphism and the morphisms $h_1,\cdots,h_5$ described below.
    \[
    \begin{split}
        h_1 &: s_1 \mapsto \rho_3 t_4,\, s_2 \mapsto s_3,\, t_1 \mapsto t_3,\, t_2 \mapsto \rho_{23}t_2, \\
        h_2 &: s_2 \mapsto s_1,\, s_3 \mapsto \rho_1 t_3,\, t_2 \mapsto \rho_{23}t_4, \\
        h_3 &: s_1 \mapsto \rho_1 t_2, \\
        h_4 &: s_2 \mapsto \rho_1 t_2, \\
        h_5 &: s_3 \mapsto \rho_3 t_1.
    \end{split}
    \]
    It is straightforward to see that $h_1,\cdots,h_5$ are not degree-preserving. Therefore the space of degree-preserving type-D endomorphisms of $\widehat{CFD}(S^3 \setminus T)$ is one-dimensional and generated by the identity morphism.
\end{proof}

We need one more lemma regarding the $\iota_K$-local equivalence class of the involutive knot Floer homology of $\mathrm{Wh}^+(T)$, which we will denote as $D$ for simplicity.

\begin{lem}\label{lem:iotaKD}
The $\iota_K$-complex of $D$ is $\iota_K$-locally equivalent to the knot Floer complex of the twist knot $5_2$, or equivalently, $T \# E$.
\end{lem}

\begin{proof}
It follows from the proof of \cite[Lemma A.1]{HeddenKimLiv} that we have 
\[
    CFK_{UV}(S^3,D) \simeq CFK_{UV}(S^3,T)\oplus A^{\oplus 3},
\]
where $A$ is a unit box summand, as shown below.
\[
\xymatrix{
\beta \ar[dd]_{V} & & \alpha \ar[ll]_{U}\ar[dd]^{V} \\
& &  \\
\delta & & \gamma \ar[ll]^{U}
}
\]
We follow the general strategy of the proof of \cite[Proposition 8.1]{HendricksManolescu}, noting that in this case, our knot is not thin. In particular, the proof uniquely determines $\iota_K \mod (U,V)$ up to chain homotopy, which we summarize in the following table. Note that we have three copies of $A$; we denote their generators as $\alpha_i,\beta_i,\gamma_i,\delta_i$ for $i=0,1,2$.
\begin{center}
\begin{tabular}{*{12}{@{\hspace{10pt}}c}}
\hline
& && $\partial$ && $\gr_U$ && $\gr_V$ && $A$ && $\iota_K \mod (U,V)$ \\
\hline
& $\rho$ && $0$ && $0$ && $-2$ && $1$ && $\tau$ \\ 
& $\sigma$ && $U\rho+ V \tau$ && $-1$ && $-1$ && $0$ && $\sigma+\delta_0$ \\ 
& $\tau$ && $0$ && $-2$ && $0$ && $-1$ && $\rho$ \\ 
& $\alpha_0$ && $U\beta_0+ V \gamma_0$ && $-1$ && $-1$ && $0$ && $\alpha_0+\sigma$ \\ 
& $\beta_0$ && $V \delta_0$ && $0$ && $-2$ && $1$ && $\gamma_0 +\tau$ \\ 
& $\gamma_0$ && $U \delta_0$ && $-2$ && $0$ && $-1$ && $\beta_0+\rho$ \\ 
& $\delta_0$ && $0$ && $-1$ && $-1$ && $0$ && $\delta_0$ \\ 
& $\alpha_1$ && $U\beta_i+ V \gamma_1$ && $-2$ && $-2$ && $0$ && $\alpha_2$ \\ 
& $\beta_1$ && $V \delta_1$ && $-1$ && $-3$ && $1$ && $\gamma_2$ \\ 
& $\gamma_1$ && $U \delta_1$ && $-3$ && $-1$ && $-1$ && $\beta_2$ \\ 
& $\delta_1$ && $0$ && $-2$ && $-2$ && $0$ && $\delta_2$ \\ 
& $\alpha_2$ && $U\beta_i+ V \gamma_2$ && $-2$ && $-2$ && $0$ && $\alpha_1+\delta_1$ \\ 
& $\beta_2$ && $V \delta_2$ && $-1$ && $-3$ && $1$ && $\gamma_1$ \\ 
& $\gamma_2$ && $U \delta_2$ && $-3$ && $-1$ && $-1$ && $\beta_1$ \\ 
& $\delta_2$ && $0$ && $-2$ && $-2$ && $0$ && $\delta_1$ \\ 
\hline
\end{tabular}
\end{center}

We are left with the possibility that $\iota_K$ may contain terms that are multiples of $U$ or $V$. Due to the bidegree reasons, such terms can arise only for the values of $\iota_K$ at $\alpha_i,\beta_i,\gamma_i,\delta_i$ for $i=1,2$.
Since $\partial \alpha_1 = U\beta_1 + V\gamma_1$, we may homotope $\iota_K$, which changes the values of $\iota_K$ only at $\alpha_1,\beta_1,\gamma_1$, to eliminate all terms arising in $\iota_K(\alpha_1)$ which have nontrivial $U$ or $V$ exponents in their coefficients. Then, since $\partial \beta_1 = V\delta_1$, we can further homotope $\iota_K$, which changes the values of $\iota_K$ only at $\beta_1,\gamma_1,\delta_1$, to eliminate all terms arising in $\iota_K(\beta_1)$ involving nontrivial $U$-exponents. Thus we may assume that 
\[
\iota_K(\alpha_1) = \alpha_2 \quad\text{and}\quad \iota_K(\beta_1) = \gamma_2 + Vc
\]
for some element $c$. But then $c$ should lie in bidegree $(-3,1)$, and there is no such an element in the given complex, so we should have 
\[
\iota_K(\beta_1) = \gamma_2.
\]
Then it follows from 
\[
V\gamma_2 + U\iota_K(\gamma_1) = \iota_K(\partial \alpha_1) = \partial \iota_K(\alpha_1) = U\beta_2 + V\gamma_2
\]
that $\iota_K(\gamma_1) = \beta_2$, and then we deduce from
\[
U\iota_K(\delta_1) = \iota_K(\partial \beta_1) = \partial \iota_K(\beta_1) = U\delta_2
\]
that $\iota_K(\delta_1) = \delta_2$. Thus, to summarize, we have homotoped $\iota_K$ in the square summand generated by $\{\alpha_1,\beta_1,\gamma_1,\delta_1\}$ so that it acts by
\[
\alpha_1\mapsto \alpha_2,\quad \beta_1\mapsto \gamma_2,\quad \gamma_1\mapsto \beta_2,\quad \delta_1 \mapsto \delta_2.
\]
Similarly, we can also homotope $\iota_K$ in the square summand generated by $\{\alpha_2,\beta_2,\gamma_2,\delta_2\}$ so that it acts by 
\[
\alpha_2 \mapsto \alpha_1+\delta_1,\quad \beta_2 \mapsto \gamma_1,\quad \gamma_2\mapsto \beta_1,\quad \delta_2\mapsto \delta_1.
\]
But then the action of $\iota_K$ splits as the direct sum of its action on the summand generated by $\rho,\sigma,\tau,\alpha_0,\beta_0,\gamma_0,\delta_0$ and the summand generated by $\alpha_i,\beta_i,\gamma_i,\delta_i$, $i=1,2$. Since the latter summand is acyclic after localizing by $(U,V)^{-1}$, we see that the given $\iota_K$-complex is $\iota_K$-locally equivalent to its summand generated by $\rho,\sigma,\tau,\alpha_0,\beta_0,\gamma_0,\delta_0$. However \cite[Proposition 8.1]{HendricksManolescu} implies that this summand is $\iota_K$-locally equivalent to the involutive knot Floer complex of $T\# E$. The lemma follows.
\end{proof}

We are now ready to prove \Cref{thm:main}; its proof will be divided in two parts. In the first part, we will prove that the knots 
\[
\left\{K_n = (D\# E)_{2n+1,-1} \# -D_{2n+1,-1} \# -E_{2n+1,-1} \mid n>0 \right\}
\]
have infinite order in the smooth concordance group. Then, in the second part, we will prove that infinitely many $K_n$ form a linearly independent family of the smooth concordance group.

\begin{proof}[Proof of \Cref{thm:main}, first part]
It follows from \Cref{lem:iotaKD} and \cite[Theorem 1.2]{kang2022involutive} that 
    \[
    [(D\# E)_{2n+1,-1}] = [(T\# E\# E)_{2n+1,-1}] = [T_{2n+1,-1}] \in \mathfrak{I}^U_K.
    \]
Since $CFK_{UV}(S^3,D)$ has $CFK_{UV}(S^3,T)$ as a direct summand, the type-D module $\widehat{CFD}(S^3 \setminus D)$ also has $\widehat{CFD}(S^3 \setminus T)$ as a direct summand; here, all knot complements are endowed with the 0-framing. Take the inclusion map
    \[
    i_T \colon \widehat{CFD}(S^3 \setminus T) \rightarrow \widehat{CFD}(S^3 \setminus D).
    \]
    By box-tensoring with the identity morphism on the type-DA module $CFDA(T_\infty \setminus P_{2n+1,-1})$, we get a type-D morphism
    \[
    (i_T)_{2n+1,-1}\colon \widehat{CFD}(S^3 \setminus T_{2n+1,-1}) \rightarrow \widehat{CFD}(S^3 \setminus D_{2n+1,-1}).
    \]
    Box-tensoring this morphism further with the identity morphism on $CFA^-(T_\infty,\nu)$ gives a $(n_z,n_w)$-bidegree-preserving chain map 
    \[
    \mathrm{Minus}\left((i_T)_{2n+1,-1}\right)\colon CFK^-(S^3,T_{2n+1,-1})\rightarrow CFK^-(S^3,D_{2n+1,-1}),
    \]
    whose truncation by $U=0$ is given by $\mathrm{Hat}((i_T)_{2n+1,-1})$. Then, by \cite[Lemma 3.1]{kang2022torsion} and \Cref{lem:KangParkLem2}, it follows that
    \[
    \iota_{D_{2n+1,-1}} \circ \mathrm{Hat}\left((i_T)_{2n+1,-1}\right) \circ \iota^{-1}_{T_{2n+1,-1}} \sim \mathrm{id}_{\widehat{CFA}(T_\infty,P_{2n+1,-1})}\boxtimes \left(\iota_{S^3 \setminus D}\circ \left(\mathrm{id}_{\widehat{CFDA}(\mathbf{AZ})}\boxtimes  i_T\right)\circ \iota^{-1}_{S^3 \setminus T}\right)
    \]
    for some bordered involutions $\iota_{S^3 \setminus D}\in \mathbf{Inv}_D(S^3 \setminus D)$ and $\iota_{S^3 \setminus T}\simeq \mathbf{Inv}_D(S^3 \setminus T)$, since we have 
    \[
    \widehat{CFA}(T_\infty,P_{2n+1,-1}) \simeq \widehat{CFA}(T_\infty,\nu)\boxtimes \widehat{CFDA}(T_\infty \setminus P_{2n+1,-1}).
    \]
    Hence, if we consider the type-D morphism
    \[
    F = i_T + \iota_{S^3 \setminus D}\circ \left(\mathrm{id}_{\widehat{CFDA}(\mathbf{AZ})}\boxtimes  i_T\right)\circ \iota^{-1}_{S^3 \setminus T},
    \]
    then we have 
    \[
    \mathrm{Hat}((i_T)_{2n+1,-1}) + \iota_{D_{2n+1,-1}} \circ \mathrm{Hat}((i_T)_{2n+1,-1}) \circ \iota^{-1}_{T_{2n+1,-1}} \sim \mathrm{id}_{\widehat{CFA}(T_\infty,P_{2n+1,-1})} \boxtimes F.
    \]

    \textbf{Claim: $\mathrm{id}_{\widehat{CFA}(T_\infty,P_{2n+1,-1})} \boxtimes F$ is nullhomotopic.} Assuming the claim, we immediately get
    \[
    \mathrm{Hat}((i_T)_{2n+1,-1}) \circ \iota_{T_{2n+1,-1}} \sim \iota_{D_{2n+1,-1}} \circ \mathrm{Hat}((i_T)_{2n+1,-1}).
    \]
    To prove the claim, we recall that $CFK_{UV}(S^3,D)\simeq CFK_{UV}(S^3,T)\oplus A^{\oplus 3}$; note that the three $A$ summands lie in different bigradings, but we do not have to care about this issue here. By \cite[Theorem 11.26]{lipshitz2018bordered}, this corresponds to the splitting 
    \[
    \widehat{CFD}(S^3 \setminus D)\simeq \widehat{CFD}(S^3 \setminus T) \oplus S^{\oplus 3},
    \]
    where the inclusion map for the $\widehat{CFD}(S^3 \setminus T)$ summand is the map $i_T$. Choose any one of the three $S$-summands, and take its projection map
    \[
    \mathrm{pr}_S:\widehat{CFD}(S^3 \setminus D) \rightarrow S.
    \]
    Then it follows from \Cref{lem:tensorlem} that $\mathrm{id}_{\widehat{CFA}(T_\infty,P_{2n+1,-1})}\boxtimes \left(\mathrm{pr}_S \circ F\right)$ is nullhomotopic. 

    It remains to show that $\mathrm{id}_{\widehat{CFA}(T_\infty,P_{2n+1,-1})} \boxtimes (\mathrm{pr}_T \circ F)$ is nullhomotopic, were $\mathrm{pr}_T$ denotes the projection map to the trefoil summand, i.e.
    \[
    \mathrm{pr}_T:\widehat{CFD}(S^3 \setminus D)\rightarrow \widehat{CFD}(S^3 \setminus T).
    \]
    We will actually prove a stronger assertion that $\mathrm{pr}_T\circ F$ is nullhomotopic. To see this, we start with the fact that $F$ is degree-preserving, which implies that $\mathrm{pr}_T \circ F$ is a degree-preserving type-D endomorphism of $\widehat{CFD}(S^3 \setminus T)$. Then it follows from \Cref{lem:trefoillem} that $\mathrm{pr}_T \circ F$ is either nullhomotopic or homotopic to the identity morphism. However, if it were homotopic to the identity, then we should have 
    \[
    \begin{split}
        \mathrm{id}_{\widehat{CFK}(S^3,T)} &= \mathrm{id}_{\widehat{CFA}(T_\infty,\nu)} \boxtimes \left(\mathrm{pr}_T \circ F\right) \\
        &\sim \mathrm{pr}_{\widehat{CFK}(S^3,T)} \circ \left(\mathrm{id}_{\widehat{CFA}(T_\infty,\nu)} \boxtimes F\right) \\
        &\sim \mathrm{pr}_{\widehat{CFK}(S^3,T)} \circ \left(\mathrm{Hat}((i_T)_{2n+1,-1}) + \iota_{D_{2n+1,-1}} \circ \mathrm{Hat}\left((i_T)_{2n+1,-1}\right) \circ \iota^{-1}_{T_{2n+1,-1}}\right),
    \end{split}
    \]
    where $\mathrm{pr}_{\widehat{CFK}(S^3,T)}$ denotes the projection map from $\widehat{CFK}(S^3 ,D)$ to its direct summand $\widehat{CFK}(S^3,T)$. However, the right hand side of the above equation is the truncation (by $U=V=0$) of a chain endomorphism of $CFK_{UV}(S^3,T)$ which becomes nullhomotopic after localizing by $(U,V)^{-1}$. This is impossible since the identity map of $\widehat{CFK}(S^3,T)$ clearly does not satisfy this property. Thus $\mathrm{pr}_T \circ F$ should be nullhomotopic and the claim is proved.
    
    Now that we have
    \[
    \mathrm{Hat}((i_T)_{2n+1,-1}) \circ \iota_{T_{2n+1,-1}} \sim \iota_{D_{2n+1,-1}} \circ \mathrm{Hat}((i_T)_{2n+1,-1}),
    \]
    it follows that $\mathrm{Minus}((i_T)_{2n+1,-1})$ is an almost $\iota_K$-local map by \cite[Lemma 3.3]{kang2022torsion}, i.e.\ we have an inequality
    \begin{equation}\label{eq:TleqD}
    [T_{2n+1,-1}] \le [D_{2n+1,-1}].
    \end{equation}
    Since $[(D\# E)_{2n+1,-1}] = [T_{2n+1,-1}] \in \mathfrak{I}^U_K$ and $K_n = (D\# E)_{2n+1,-1} \# -D_{2n+1,-1}\# -E_{2n+1,-1}$, we get
    \[
    [K_n] = [T_{2n+1,-1}] - [D_{2n+1,-1}] - [E_{2n+1,-1}] \le -[E_{2n+1,-1}].
    \]
    Furthermore, it is shown in the proof of \cite[Theorem 4.5]{kang2022torsion} that $0 < [E_{2n+1,-1}]$ and $[E_{2n+1,-1}]$ has infinite order in $\mathfrak{I}^U_K$. Hence, we deduce that 
    \[
    [K_n] \le -[E_{2n+1,-1}] < 0,
    \]
    and $[K_n]$ has infinite order in $\mathfrak{I}^U_K$. Therefore $K_n$ has infinite order in the smooth concordance group.
\end{proof}

To prove the linear independence part of \Cref{thm:main}, we recall more facts about $\mathfrak{I}^U_K$ from \cite{kang2022torsion}. Recall from \cite[Section 4]{kang2022torsion} that, for each $n\ge 2$, the horizontal almost $\iota_K$-complex $C_n$ is generated by elements $a_n,b_n,c_n,d_n,x_n$, where $a_n$ and $x_n$ have bidegree $(0,0)$. The differential is given by
\[
\partial a_n = U^n b_n,\,\partial c_n = U^n d_n, \, \partial b_n = \partial d_n = \partial x_n = 0
\]
and the involution on its hat-flavored truncation $\widehat{C}_n$ is given by
\[
\iota_{C_n}(a_n)=a_n+x_n,\,\iota_{C_n}(b_n) = c_n,\,\iota_{C_n}\text{ fixes } d_n,x_n.
\]

These complexes satisfy the following inequalities in $\mathfrak{I}^U_K$ (for detailed discussion see the proofs of \cite[Lemma 4.6 and Theorem 4.7]{kang2022torsion}). First of all, for any $n\ge 2$, we have 
\begin{equation}\label{eq:CnvsCnplus1}
0<[C_n] \quad \text{and} \quad M \cdot [C_n] < [C_{n+1}] \quad \text{for any integer $M$}.    
\end{equation}
Furthermore, for any $n\ge 2$, we have 
\begin{equation}\label{eq:CnvsEn}
    [C_n] \le [E_{2n+1,-1}].
\end{equation}
Lastly, if a knot $J$ satisfies $CFK_{UV}(S^3,J) \simeq \mathbb{F}_2[U,V]\oplus A$, for some acyclic summand $A$ (which is the case for each knot $K_n$), then there is a positive integer $N\geq 2$ such that
\begin{equation}\label{eq:upperboundonJ}
    M\cdot [J] < [C_N] \quad \text{for any integer $M$}.
\end{equation}
Now, we are ready to prove the second part of the main theorem.

\begin{proof}[Proof of \Cref{thm:main}, second part]
    It follows from the first part of the proof of \Cref{thm:main} and \eqref{eq:CnvsEn} that for each $n\ge 2$, we have 
    \[
    [K_n] \le -[E_{2n+1,-1}] \le -[C_n].
    \]
    Furthermore, by \eqref{eq:upperboundonJ} for each positive integer $n$, there is a positive integer $N(n)$ such that 
    \[
    -[C_{N(n)}] < M\cdot [K_n] \quad \text{for any intger } M.
    \]
    Hence, if we choose an increasing sequence of integers $\{ s_n \}$ 
such that $s_1 = 2$ and $s_{n+1} \ge N(s_n)$ for each~$n\ge 1$, then by combining the above two inequalities with \eqref{eq:CnvsCnplus1} we have that 
    \[
    [K_{s_{n+1}}] \le -[C_{s_{n+1}}] \le -[C_{N(s_n)}] < M\cdot [K_{s_n}] \quad \text{for any intger } M.
    \]

    Suppose that the knots $K_{s_n}$ are linearly dependent in the smooth concordance group. Then there exists a finite sequence $D_1, D_2, \ldots, D_n$ of nonzero integers such that
    \[
    D_1 \cdot [K_{s_1}] + D_2 \cdot [K_{s_2}]  +\cdots + D_n \cdot [K_{s_n}] = 0.
    \]
    Without loss of generality, we may assume that $D_n > 0$. Then we have 
    \[
    \begin{split}
        0 &= D_1 \cdot [K_{s_1}] + D_2 \cdot [K_{s_2}]  +\cdots + D_n \cdot [K_{s_n}] \\
        &< (D_2 - 1)\cdot [K_{s_2}] + D_3 \cdot [K_{s_3}] + \cdots + D_n\cdot [K_{s_n}] \\
        &< (D_3 - 1)\cdot [K_{s_3}] + D_4\cdot [K_{s_4}] + \cdots + D_n\cdot [K_{s_n}] \\
        &< \cdots < (D_n - 1) \cdot  [K_{s_n}] \leq 0,
    \end{split}
    \]
    a contradiction. Therefore the knots $K_{s_n}$ are linearly independent, as desired.
\end{proof}

\begin{rem}
By applying the arguments used in the proof of \cite[Lemma 3.2]{HKPS}, it follows that we can take $N(n)$ to be $2n+2$. With a more explicit calculation of $\iota_K$ for $K_n$, we expect that a much smaller value for $N(n)$ should be possible.\end{rem}

\begin{rem}
    The argument used in the second part of the proof of \Cref{thm:main} can be summarized as follows; note that the same argument was also used in \cite{kang2022torsion}. Let $\{ K_n \}$ be a sequence of rationally slice knots such that for any $i\ge 2$, there exists a positive integer $n_i$ such that $[K_{n_i}] \ge [C_i]$ in $\mathfrak{I}^U_K$. Then $\{ K_n \}$ admits a linearly independent infinite subsequence in $\mathcal{C}$.
\end{rem}

\bibliographystyle{amsalpha}
\bibliography{ref}
\end{document}